\documentclass[12pt, reqno]{amsart}
\usepackage{amsmath, amsthm, amscd, amsfonts, amsthm, amssymb, graphicx,lineno,xcolor}
\usepackage[pagebackref=false, bookmarksnumbered, colorlinks, plainpages, linkcolor=blue, citecolor=magenta]{hyperref}

\newtheorem{theorem}{Theorem}[section]
\newtheorem{lemma}[theorem]{Lemma}
\newtheorem{proposition}[theorem]{Proposition}

\theoremstyle{definition}

\newtheorem{example}[theorem]{Example}

\newtheorem{remark}[theorem]{Remark}
\numberwithin{equation}{section}

\begin{document}

\title[Further properties of involutory and idempotent matrices]{Further properties of involutory and idempotent matrices}

\author[F. Mirzapour]{ F. Mirzapour}

\address{ Department of Mathematics, Faculty of Sciences, University of Zanjan, P.O. Box 45371-38791, Zanjan, Iran.}

\email{f.mirza@znu.ac.ir}

\author[A. Mirzapour]{A. Mirzapour}
\address{Department of Electrical and Computer, Zanjan branch, Islamic Azad University, Zanjan, Iran}

\email{aminmirzapour60@gmail.com}

\subjclass[2010]{Primary 15A24; Secondary 65F30.}

\keywords{root, involutary, idempotent}

\begin{abstract}
In this paper, we will derive the real roots of certain sets of matrices with real entries. We will also demonstrate that real orthogonal matrices can have real root or be involutory. Eventually, we will represent idempotent matrices in a block form.
\end{abstract}
\maketitle


\section{Introduction  and Preliminaries}
Let $\mathcal{M}_n$ denote the $C^*-$algebra of all $n-$square matrices. Following is the list of special subsets of  $\mathcal{M}_n$:
\begin{itemize}
\item[•] $\mathcal{O}_n$: The group of orthogonal matrices,
\item[•] $\mathcal{U}_n$: the group of all unitary matrices,
\item[•] $\mathcal{S}_n$: the set of symmetric matrices,
\item[•] $\mathcal{I}_n$: the set of all involutory matrices, 
\item[•] $\mathcal{P}_n$: the set of all idempotent matrices.
\end{itemize}
We know that two matrices $A$ and $B$ are similar, if there exists an invertible matrix T such that $A=T^{-1}BT$, and it is diagonalizable if there exist $\lambda_1,\ldots,\lambda_n$ such that $A=T^{-1}\mathrm{diag}(\lambda_1,\ldots,\lambda_n)T$ and unitarily diagonalizable if $T=U\in\mathcal{U}_n$, $A=U^*\mathrm{diag}(\lambda_1,\ldots,\lambda_n)U$. From the other side, matrix $C$ is a root of $A$, if $A=C^2$, and it can also be said that $A$ is root - approximable if there exists a sequence $\{C_k\}$ such that $C_k\longrightarrow I$ and $C_k^{2^k}=A$, for each $k=0,1,2,\ldots$ \cite{Chad4}, \cite{Magh7}. Matrix functions have been studied in \cite{Bhat6}, \cite{Horn9}, \cite{Nich2}. In this paper, the matrix function $f(A)=\sqrt{A}$ for specific matrices is studied.\\    
The square matrix $A$ is said to be idempotent or projection, if
$A^2 = A,$
and involutory if
$A^2 = I.$
In this article, we need the following propositions which are from \cite{High3}, \cite{Zhan1}.
\begin{proposition}
Let $A$ be an $n-$square complex matrix. Then
\begin{itemize}
\item[1.] $A$ is idempotent if and only if $A$ is similar to a diagonal matrix
of the form $\mathrm{diag}(1, . . . , 1, 0, . . . , 0)$.
\item[2.] A is involutory if and only if A is similar to a diagonal matrix
of the form $\mathrm{diag}(1, . . . , 1,-1, . . . ,-1)$.
\end{itemize}
\end{proposition}

\begin{proposition}
Let $A$ and $B$ be real square matrices of the same
size. If $P$ is a complex invertible matrix such that $P^{-1}AP = B$, then
there exists a real invertible matrix $Q$ such that $Q^{-1}AQ = B$.
\end{proposition}

\begin{proposition}
Let $A$ and $B$ be real square matrices of the same
size. If $A = UBU^*$ for some unitary matrix $U$, then there exists a
real orthogonal matrix $Q$ such that $A = QBQ^T$.
\end{proposition}

\begin{proposition}
Every real orthogonal matrix is real orthogonally similar
to a direct sum of real orthogonal matrices of order at most $2$.
\end{proposition}
\section{Involution matrices}
In this section, we obtain some properties of the involutory matrices and by applying them we derive the real root of some special matrices. We start with matrices of order 2.
\begin{lemma} 
The class of all real involutory matrices of order 2 is as follows:
\begin{align*}
\left\{ \left(\begin{array}{cc}a & b\\ \frac{1-a^2}{b} & -a\end{array}\right);~ a,b\in\mathbb{R},~b\neq 0\right\}\bigcup \left\{\left(\begin{array}{cc}\pm 1 & 0\\ c & \mp 1\end{array}\right);~ c\in\mathbb{R}\right\}\bigcup\left\{ \pm I_2\right\}.
\end{align*}
\end{lemma}
\begin{lemma} 
The class of all real matrices like $A$ such that $A^2=-I_2$ is as follows:
\begin{align*}
\left\{ \left(\begin{array}{cc}a & -b\\ \frac{1+a^2}{b} & -a\end{array}\right);~ a,b\in\mathbb{R},~b\neq 0\right\}.
\end{align*}
\end{lemma}
Herein after we refer to 
$
\left(\begin{array}{cc}a & -b\\ \frac{1+a^2}{b} & -a\end{array}\right)
$
as $\Psi(a,b)$.
\begin{remark}
In the lemma (2.1), if $|a|\leq 1$ and $b=\sqrt{1-a^2}$, then 
\begin{align*}
\left(\begin{array}{cc}\cos\theta & \sin\theta\\ \sin\theta & -\cos\theta\end{array}\right);~~\theta\in\mathbb{R}.
\end{align*}
\end{remark}
\begin{remark}
$\left(\begin{array}{cc}\pm 1 & 0\\ 0 & \pm i\end{array}\right)$ are the only roots of $\left(\begin{array}{cc} 1 & 0\\ 0 & -1\end{array}\right)$.
\end{remark}
\begin{theorem}
Suppose $A$ is a real involutory matrix of order $n$ and $\det A>0$, then $A$ has a real root. 
\end{theorem}
\begin{proof}
 Since $A$ is a real involutory matrix, then by propositions (1.1) and (1.2), there is an invertible real matrix $B$ such that 
$$
A=B^{-1}\mathrm{diag}(1,\ldots,1,-1,\ldots,-1)B,
$$
 thus $\det A=1$ or $\det A=-1$. If $\det A=1$, then the number of eigenvalues $-1$ is even. Therefore 
\begin{align*} 
 C&=\mathrm{diag}(1,\ldots,1,-1,\ldots,-1)\\
 &=I_k\oplus\left(\begin{array}{cc}-1 & 0\\ 0 & -1 \end{array}\right)\oplus\cdots\oplus\left(\begin{array}{cc}-1 & 0\\ 0 & -1 \end{array}\right),
\end{align*}  
 has many real roots, for instance, for arbitrary real numbers $a_1,\ldots, a_t$ and non-zero real numbers $b_1,\ldots, b_t$, if
\begin{align*} 
D=\mathrm{diag}(\pm 1,\cdots,\pm 1)\oplus\Psi(a_1,b_1)\oplus\cdots\oplus\Psi(a_t,b_t),
\end{align*}
then $A=B^{-1}D^2B=(B^{-1}DB)^2$. 
\end{proof}
\begin{theorem}
Let $A$ be a real symmetric matrix of size $n$ with the eigenvalues $\lambda_1\geq\cdots\geq\lambda_n$. If every negative eigenvalue is repeated twice or even times, then $A$ has a real root. 
\end{theorem}
\begin{proof}
 Given $A$ is real symmetric, then by proposition (1.3) there is a real orthogonal matrix $Q$ such that $A=Q^T\mathrm{diag}(\lambda_1,\ldots,\lambda_n)Q$. If $\lambda_n\geq 0$, we have nothing to prove. Now if 
$$
\lambda_1\geq\cdots\geq\lambda_k\geq0>\lambda_{k+1}\geq\cdots\geq\lambda_n,
$$
then by assumption, each $\lambda_j,~~k+1\leq j\leq n $ is repeated twice or even times, therefore 
\begin{align*} 
C=\mathrm{diag}(\lambda_1,\cdots,\lambda_k)\oplus\left(\begin{array}{cc}\lambda_{k+1} & 0\\ 0 &\lambda _{k+1}\end{array}\right)\oplus\cdots\oplus\left(\begin{array}{cc}\lambda_n & 0\\ 0 & \lambda_n \end{array}\right).
\end{align*}  
Thus for all real numbers $a_{k+1},\ldots,a_n$ and non-zero real numbers $b_{k+1},\ldots,b_n$,
\begin{align*} 
D=\mathrm{diag}(\pm\sqrt{\lambda_1},\cdots,\pm\sqrt{\lambda_k})\oplus\sqrt{-\lambda_{k+1}}\Psi(a_{k+1},b_{k+1})
\oplus\cdots\oplus\sqrt{-\lambda_n}\Psi(a_n,b_n),
\end{align*} 
are real roots of $C$ and  $A=(Q^TDQ)^2$.
\end{proof}
In this case, note that this matrix has no symmetrical root.\\
According to proposition (1.4), for any real orthogonal matrix $A$ there exist $\alpha_1,\ldots,\alpha_s$ and $\beta_1\ldots,\beta_t$ such that $A$ is similar to matrix
\begin{align*} 
C&= I_k\oplus-I_l\oplus \left(\begin{array}{cc}\cos\alpha_1 & \sin\alpha_1\\ -\sin\alpha_1 & \cos\alpha_1 \end{array}\right)\oplus\cdots\oplus\left(\begin{array}{cc}\cos\alpha_s & \sin\alpha_s\\ -\sin\alpha_s & \cos\alpha_s \end{array}\right)\\
& \oplus\left(\begin{array}{cc}\cos\beta_1 & \sin\beta_1\\ \sin\beta_1 & -\cos\beta_1 \end{array}\right)\oplus\cdots\oplus\left(\begin{array}{cc}\cos\beta_t & \sin\beta_t\\\sin\beta_t & -\cos\beta_t \end{array}\right).
\end{align*}
We have the following theorem for real orthogonal matrices.
\begin{theorem}
Suppose $A$ is a real orthogonal matrix that is similar to the following matrix
\begin{align*} 
C&= I_k\oplus-I_l\oplus \left(\begin{array}{cc}\cos\alpha_1 & \sin\alpha_1\\ -\sin\alpha_1 & \cos\alpha_1 \end{array}\right)\oplus\cdots\oplus\left(\begin{array}{cc}\cos\alpha_s & \sin\alpha_s\\ -\sin\alpha_s & \cos\alpha_s \end{array}\right),
\end{align*}
where the eigenvalue $-1$ is repeated even times,  then $A$ has a real orthogonal root.
\end{theorem}
\begin{proof}
 By the assumption, there is a real orthogonal matrix $P$ such that 
\begin{align*} 
C&=P^TAP=I_k\oplus-I_l\oplus \left(\begin{array}{cc}\cos\alpha_1 & \sin\alpha_1\\ -\sin\alpha_1 & \cos\alpha_1 \end{array}\right)\oplus\cdots\oplus\left(\begin{array}{cc}\cos\alpha_s & \sin\alpha_s\\ -\sin\alpha_s & \cos\alpha_s \end{array}\right),\\
&= I_k\oplus \left(\begin{array}{cc}\cos\pi & \sin\pi\\ -\sin\pi & \cos\pi \end{array}\right)\oplus\cdots\oplus\left(\begin{array}{cc}\cos\pi & \sin\pi\\ -\sin\pi & \cos\pi \end{array}\right)\\
& \oplus \left(\begin{array}{cc}\cos\alpha_1 & \sin\alpha_1\\ -\sin\alpha_1 & \cos\alpha_1 \end{array}\right)\oplus\cdots\oplus\left(\begin{array}{cc}\cos\alpha_s & \sin\alpha_s\\ -\sin\alpha_s & \cos\alpha_s \end{array}\right),\\
\end{align*}
if
\begin{align*}
D&= \mathrm{diag}(\pm 1,\cdots,\pm 1)\oplus \left(\begin{array}{cc}\cos\frac{\pi}{2} & \sin\frac{\pi}{2}\\ -\sin\frac{\pi}{2} & \cos\frac{\pi}{2} \end{array}\right)\oplus\cdots\oplus\left(\begin{array}{cc}\cos\frac{\pi}{2} & \sin\frac{\pi}{2}\\ -\sin\frac{\pi}{2} & \cos\frac{\pi}{2} \end{array}\right)\\
& \oplus \left(\begin{array}{cc}\cos\frac{\alpha_1}{2} & \sin\frac{\alpha_1}{2}\\ -\sin\frac{\alpha_1}{2} & \cos\frac{\alpha_1}{2} \end{array}\right)\oplus\cdots\oplus\left(\begin{array}{cc}\cos\frac{\alpha_s}{2} & \sin\frac{\alpha_s}{2}\\ -\sin\frac{\alpha_s}{2} & \cos\frac{\alpha_s}{2} \end{array}\right),
\end{align*}
then $A=PCP^T=PD^2P^T=(PDP^T)^2$, where $PDP^T$ is real orthogonal root of $A$.
\end{proof}
\begin{remark}
If 
\begin{align*}
D_k&= I_r\oplus \left(\begin{array}{cc}\cos\frac{\pi}{2^k} & \sin\frac{\pi}{2^k}\\ -\sin\frac{\pi}{2^k} & \cos\frac{\pi}{2^k} \end{array}\right)\oplus\cdots\oplus\left(\begin{array}{cc}\cos\frac{\pi}{2^k} & \sin\frac{\pi}{2^k}\\ -\sin\frac{\pi}{2^k} & \cos\frac{\pi}{2^k} \end{array}\right)\\
&\oplus \left(\begin{array}{cc}\cos\frac{\alpha_1}{2^k} & \sin\frac{\alpha_1}{2^k}\\ -\sin\frac{\alpha_1}{2^k} & \cos\frac{\alpha_1}{2^k} \end{array}\right)\oplus\cdots\oplus\left(\begin{array}{cc}\cos\frac{\alpha_s}{2^k} & \sin\frac{\alpha_s}{2^k}\\ -\sin\frac{\alpha_s}{2^k} & \cos\frac{\alpha_s}{2^k} \end{array}\right),
\end{align*}
then $D_k\longrightarrow I$ and $(PD_kP^T)^{2^k}=A$, i.e. $A$ is root-approximable.
\end{remark}
\begin{remark}
With the above given if 
 \begin{align*} 
P^TAP&= I_r\oplus -I_l\oplus \left(\begin{array}{cc}\cos\beta_1 & \sin\beta_1\\ \sin\beta_1 & -\cos\beta_1 \end{array}\right)\oplus\cdots\oplus\left(\begin{array}{cc}\cos\beta_t & \sin\beta_t\\\sin\beta_t & -\cos\beta_t \end{array}\right),
\end{align*} 
then $A$ is an involutory matrix.
\end{remark}
\section{Idempotent matrices}
By proposition (1.1), if $P$ is an idempotent  matrix, then it is similar to $\left(\begin{array}{cc}I & O\\ O & O \end{array}\right)$ where $I$ is identity, i.e. there are matrices $A, B, C$ and $D$ such that $A$ and $D$ are square and $A$ and $I$ are of the same size, then  $M=\left(\begin{array}{cc}A & B\\ C & D \end{array}\right)$ is invertible and 
 $$
P=M\left(\begin{array}{cc}I & O\\ O & O \end{array}\right)M^{-1}.
$$
If  $M^{-1}=\left(\begin{array}{cc}X & Y\\ U & V \end{array}\right)$ and $A$ is invertible, then we have 
\begin{align*}
X&=A^{-1}+A^{-1}B(D-CA^{-1}B)^{-1}CA^{-1}, \quad Y=-A^{-1}B(D-CA^{-1}B)^{-1},\\
U&=-(D-CA^{-1}B)^{-1}CA^{-1}, \quad V=(D-CA^{-1}B)^{-1},
\end{align*}
\begin{align*}
P&=\left(\begin{array}{cc}A & B\\ C & D \end{array}\right)\left(\begin{array}{cc}I & O\\ O & O \end{array}\right)\left(\begin{array}{cc}X & Y\\ U & V \end{array}\right)=\left(\begin{array}{cc}AX & AY\\ CX & CY \end{array}\right)\\
&=\left(\begin{array}{cc}I+B(D-CA^{-1}B)^{-1}CA^{-1} & -B(D-CA^{-1}B)^{-1}\\ CA^{-1}+CA^{-1}B(D-CA^{-1}B)^{-1}CA^{-1} & -CA^{-1}B(D-CA^{-1}B)^{-1} \end{array}\right),
\end{align*}
and if $D$ is invertible, we have 
 $$
A^{-1}+A^{-1}B(D-CA^{-1}B)^{-1}CA^{-1}=(A-BD^{-1}C)^{-1}.
$$
Consider $S=(D-CA^{-1}B)^{-1}$ and $T=(A-BD^{-1}C)^{-1}$. We get 
\begin{align*}
P&=\left(\begin{array}{cc}AT & -BS\\ CT & -CA^{-1}BS \end{array}\right)\\
&=\left(\begin{array}{cc}A & O\\ C & O \end{array}\right)\left(\begin{array}{cc}T & O\\ O & O \end{array}\right)-\left(\begin{array}{cc}O & B\\ O & CA^{-1}B \end{array}\right)\left(\begin{array}{cc}O & O\\ O & S \end{array}\right).
\end{align*}
Thus we have proved the following theorem:
\begin{theorem}
Let $A$ and $D$ be two invertible matrices of orders $n$ and $m$ respectively, while $B$ and $C$  are two matrices of orders $n\times m$ and $m\times n$ respectively. Then 
\begin{align*}
P=\left(\begin{array}{cc}A(A-BD^{-1}C)^{-1} & -B(D-CA^{-1}B)^{-1}\\ C(A-BD^{-1}C)^{-1} & -CA^{-1}B(D-CA^{-1}B)^{-1} \end{array}\right)
\end{align*}
is an idempotent.
\end{theorem}
\begin{example}
Let $a, b, c$ and $d$ be real numbers, whereas $bc\neq ad\neq 0$, and  
$$
A=aI_n, B=b\left(\begin{array}{cc}I_m \\ O \end{array}\right), C=c\left(\begin{array}{cc}I_m & O \end{array}\right), D=dI_m, \quad (n\geq m),
$$
\begin{align*}
M=\left(\begin{array}{ccc}aI_n & b\left(\begin{array}{cc}I_m \\ O \end{array}\right)\\ c\left(\begin{array}{cc}I_m & O \end{array}\right) & dI_m \end{array}\right),
\end{align*}
\begin{align*}
S=\frac{a}{ad-bc}I_m,\quad T=\left(\begin{array}{cc}\frac{d}{ad-bc}I_m & O\\ O & \frac{1}{a}I_{n-m} \end{array}\right),
\end{align*}
therefore 
\begin{align*}
P=\frac{1}{ad-bc}\left(\begin{array}{ccc}\left(\begin{array}{cc}{adI_m} & O\\ O & (ad-bc)I_{n-m} \end{array}\right) & \left(\begin{array}{cc}-abI_m \\ O \end{array}\right) \\  \left(\begin{array}{cc}cdI_m & O \end{array}\right) & -bcI_m \end{array}\right)
\end{align*}
is idempotent and  
\begin{align*}
T=2P-I=\frac{2}{ad-bc}\left(\begin{array}{ccc}\left(\begin{array}{cc}{\frac{ad+bc}{2} I_m} & O\\
	 O &\frac{ad-bc}{2} I_{n-m} \end{array}\right) & \left(\begin{array}{cc}-abI_m \\ 
	 O \end{array}\right) \\  \left(\begin{array}{cc}cdI_m & O \end{array}\right) & -\frac{ad+bc}{2}I_m \end{array}\right)
\end{align*}
is an involutory matrix.
\end{example}



\begin{thebibliography}{99}
\bibitem{Chad4}  A. Chademan and F. Mirzapour, {\it Midcovex functions in locally compact groups}, Proc. Amer. Math. Soc., vol 127, N. 10, 2961--2968, (1999).
\bibitem{Bhat6} R. Bhatia, {\it Matrix Analysis}, Springer-Verlag, Inc., New York, (1973).
\bibitem{Horn9} R. A Horn and C. R. Johnson, {\it Topics in Matrix Analysis.} Cambridge University Press,
(1991).
\bibitem{Magh7} S. Maghsoudi and F. Mirzapour, {\it Root-Approximability of the Group of Automorphisms
of the Unit Ball in $\mathbb{C}^n$}, Bull. Malays. Math. Sci. Soc. (2016) 39:1477–1485.
\bibitem{High3} Nicholas J. Higham. {\it Computing real square roots of a real matrix} Linear Algebra and
its Applications, 88/89:405-430, (1987).
\bibitem{Nich2} Nicholas J. Higham. {\it Functions of Matrices. Theory and Computation.} SIAM, first
edition, (2008).
\bibitem{Thomp5} R.C. Thompson, {\it The characteristic polynomial of a principal subpencil of a Hermitian matrix pencil}, Linear Algebra Appl. 14 (1976) 136–177. 
\bibitem{Thom8} R.C. Thompson, {\it Pencils of complex and real symmetric and skew symmetric matrices}, Linear Algebra Appl. 147 (1991) 323–371.
\bibitem{Zhan1}  F. Zhang, {\it Matrix Theory}, Springer-Verlag,  (2011).
\end{thebibliography}
\end{document}